\documentclass[11pt,twoside, leqno]{article}
\usepackage[left=2.25cm,right=1.97cm,top=2.25cm,bottom=2.25cm]{geometry}
\usepackage{amssymb}
\usepackage{mathrsfs}
\usepackage{array}
\usepackage{diagbox}
\usepackage{graphicx}
\usepackage{epstopdf}
\usepackage{subfigure}
\usepackage{longtable}
\usepackage{amsfonts}
\usepackage{latexsym}
\usepackage{epsfig}
\usepackage{listings}
\usepackage{float}
\usepackage{booktabs}
\usepackage{amsmath}
\usepackage{multirow}
\usepackage{amsthm}
\usepackage{color}

\newtheorem{Theorem}{\bf Theorem}[section]
\newtheorem{Lemma}{\bf Lemma}[section]

\newtheorem{Definition}{\bf Definition}[section]
\newtheorem{Proposition}{\bf Proposition}[section]

\usepackage[linesnumbered,ruled,vlined]{algorithm2e}
\usepackage{algpseudocode}
\usepackage{amsmath}
\usepackage{titlesec}

\begin{document}
	\title{Can TM system form an unconditional basis for Banach spaces?}
	\author{Haibo Yang, Chitin Hon, Qixiang Yang and Tao Qian\footnote{Corresponding author}}
	\date{\today}
	\maketitle
	\noindent \textbf{Abstract:}

The research on the algorithm of analytic signal has received much attention for a long time.
Takenaka-Malmquist (TM) system   was introduced to consider analytic functions in 1925.
If TM system satisfies hyperbolic inseparability condition,
then it is an orthogonal basis.
It can form  unconditional basis for Hilbert space $\mathbb{H}^{2}(D)$
and Schauder basis for Banach space $\mathbb{H}^{p}(D)(1 < p < \infty)$.
In characterizing a function space,
a necessary condition is whether the basis is unconditional.
But since the introduction of TM systems in 1925,
to the best of our knowledge,
no one has proved the existence of a TM system capable of forming an unconditional basis for Banach space $\mathbb{H}^{p}(D) (p \neq 2)$.
TM system has a simple and intuitive analytical structure.
Hence it is applied also to the learning algorithms
and systematically developed to the reproducing kernel Hilbert spaces (RKHS).
Due to the lack of unconditional basis properties,
it cannot be extended to the reproducing kernel Banach spaces (RKBS) algorithm.
But the case of Banach space plays an important role in machine learning.
In this paper, we prove that  two TM systems can form unconditional basis for $\mathbb{H}^{p}(D) (1<p <\infty)$.


\section{Introduction and main Theorem}
Given unit disk $D=\{z=re^{2\pi i t}, 0\leq r<1, t\in [0,1]\}$.
For $1 < p < \infty$, the space of analytic functions $\mathbb{H}^{p}(D)$ is a Banach space composed of analytic functions defined by
$$\mathbb{H}^{p}(D)= \{F(z): F(z) \text{ is analytic in disk }  D  \text{ and }
\sup\limits_{0\leq r<1} \int^{1}_{0} |F(r e^{2\pi i t})|^{p} dt<\infty\}.$$
Takenaka \cite{Ta} introduced Takenaka Malmquist (TM) system in  1925.
This rational orthogonal system generalized the Laguerre basis and two parameter Kautz basis.
TM system consists of some rational functions and blaschke products and attracted a lot of attention
\cite{AN, B}.
TM system is a non-wavelet form orthogonal system in $\mathbb{H}^{2}$.
For $\mathbb{H}^{2}$, TM systems are not only Schauder basis, but also unconditional basis.
Unconditional basis have some good advantages in data processing.
Pereverzyev \cite{P} used reproducing kernel methods to deal with learning algorithms,
and systematically developed reproducing kernel Hilbert spaces (RKHS).
But for $p\neq 2$, the situation becomes complex.
Qian-Chen-Tan \cite{QCT} and Wang-Qian \cite{WQ} proved that rational orthogonal   systems are Schauder basis for $\mathbb{H}^{p}$.
Unconditional convergence can be traced back to Cauchy's introduction of convergent series in 1821.
Since 1925,
to the best of our knowledge,
it is always an open question whether there exists a set of TM systems
which can form an unconditional basis for $\mathbb{H}^{p} (p \neq 2)$.
The unconditional basis structure in Banach space is very complex.
Until 1973, J.Hennefeld \cite{H} proved that there exist three possibility for the unconditional basis for Banach space:
(1) None basis; (2) One basis under equivalence sense; (3) uncountable number of basis.
Since we do not know whether TM system can be an unconditional basis when $p\neq 2$,
for learning algorithms,
we have been unable to use TM system in Banach space.
The objective of this paper is to consider part of the TM system
and to solve this longstanding open problem.
We prove that  some TM system can form unconditional basis in $\mathbb{H}^{p}(D)(p \neq 2)$.
Therefore, RKHS algorithm can be developed to the reproducing kernel Banach spaces (RKBS).

There are at least three methods for approximating analytic functions on the unit disk.
The first method is to use Meyer's bimodal wavelets.
See  Meyer \cite{M} and Qian-Yang \cite{QY}.
This method is obtained after some special technique processing of real-valued Meyer wavelet.
Although Meyer bimodal wavelet method can be used to obtain orthogonal basis,
it does not have a simple and intuitive analytical expression.
Further, Hon-Leong-Qian-Yang-Zou \cite{LQYZ}
recently constructed an unconditional basis of $\mathbb{H}^{p}(D)$  by using rational functions.
This unconditional basis is proved by means of wavelet and quasi-orthogonality.
Although this basis has the advantage of simple expression, it is not orthogonal basis.

The third method is to generalize the basis $1, z, z^2, \cdots$ to get TM orthogonal system \cite{Ta}.
Qian et al. \cite{QCT, WQ} rediscovered TM basis when considering greedy AFD algorithms.
They proved that TM basis all are unconditional basis in $\mathbb{H}^{2}(D)$
and Schauder basis in $\mathbb{H}^{p}(D) (1<p<\infty)$.
The object of this paper is to solve the longstanding open problem
whether there exists a TM system capable of forming an unconditional basis for Banach space $\mathbb{H}^{p}(D)(p \neq 2)$.
The $m-$th term of Takenaka-Malmquist's orthogonal system $\{B_m\}^{+\infty}_{m=1}$
is formed by the following weighted Blaschke product of the $m-1$ order:
\begin{equation}\label{eq:1.1}
B_{m}(z)= \frac{\sqrt{1-|a_m|^{2}}}{1-\bar{a}_m z} \prod\limits_{l=1}^{m-1}\frac{z-a_l}{1-\bar{a}_l z},
\end{equation}
where $a_l \,(l=1,\cdots,m,\cdots)$ are complex numbers with $|a_l|<1$ in the open unit circle.
If $a_1=\cdots=a_m=\cdots=0$, then $\{B_m(e^{2\pi i t})\}_{m=1}^{+\infty}$  is a sequence of the following form
$$ \{1, e^{2\pi i t}, \cdots, e^{2\pi i (m-1)t}, \frac{\sqrt{1-|a_{m+1}|^{2}}e^{2\pi i mt}}{1-\bar{a}_{m+1} e^{2\pi it}}, \cdots, \frac{\sqrt{1-|a_{m+r}|^{2}}e^{2\pi i mt}}{1-\bar{a}_{m+r} e^{2\pi it}} \prod\limits_{j=m+1}^{m+r-1}\frac{e^{2\pi it}-a_j}{1-\bar{a}_je^{2\pi it}}, \cdots\}.$$

In this paper, we first consider the case where all $a_l=0, (l=1,\cdots,m,\cdots)$, that is
\begin{equation}\label{eq:B1.1}
B_{m}^{(0)}(z)=\{1, z, z^2, \cdots\}.
\end{equation}

Then, we consider two cases where $a_l$ is not all $0$:

{\bf Case} (1): $a_1=0, a_l \neq 0\, (l \geq 2)$.  Then the TM system is
\begin{equation}\label{eq:B1.2}
 B_{m}^{(1)}(z)=\{1, \frac{\sqrt{1-|a_{2}|^{2}}z}{1-\bar{a}_{2} z}, \cdots, \frac{\sqrt{1-|a_{1+r}|^{2}}z}{1-\bar{a}_{1+r} z} \prod\limits_{l=2}^{r+1}\frac{z-a_l}{1-\bar{a}_lz}, \cdots\}.
\end{equation}

{\bf Case} (2): $a_1=a_2=0, a_l \neq 0 \, (l \geq 3)$. Then the TM system is
\begin{equation}\label{eq:B1.3}
B_{m}^{(2)}(z)=\{1, z, \frac{\sqrt{1-|a_{3}|^{2}}z^2}{1-\bar{a}_{3} z}, \cdots, \frac{\sqrt{1-|a_{1+r}|^{2}}z^2}{1-\bar{a}_{1+r} z} \prod\limits_{l=2}^{r+1}\frac{z-a_l}{1-\bar{a}_lz}, \cdots\}.
\end{equation}

Experience tells all the authors of this paper
that the choice of points of rational functions represents different mathematical meanings.
For example,
H. Yang and his cooperators have analyzed analytic functions, harmonic functions  and Gauss process in \cite{LQYZ, LYY, YangChen,YYH}.
{\bf The TM system is composed of smooth functions,
and we find that whether it can form an unconditional basis depends on four factors:
vanishing moment, locality, distance to boundary and rotation (distance between points).}
Because in the learning algorithm,
all that is needed is whether there are TM systems that can form an unconditional basis for Banach spaces.
Therefore, we do not consider how to obtain
the sufficient and necessary conditions of unconditional basis through these four conditions.
In this paper, a special selection is made for all $a_l$ that is not equal to 0.
In order for the system of functions to best reflect the above properties,
we reorder $m$ appropriately by distance from the boundary.
We express index $m$ in the form of two indices $(j, k)$
in order to see the distance to the boundary and the rotation of the argument.
Particularly, we take
\begin{equation}\label{eq:point}
r_{j}= \sqrt{1-2^{-j}}, h_{j,k}=2^{-j}k {\rm \, \, \text{and} \,\, }  a_m=a_{j,k}= r_{j} e^{2\pi i h_{j,k}}.
\end{equation}
For the index $m$ and indices $(j, k)$ in the above two cases \eqref{eq:B1.2} and \eqref{eq:B1.3},
their one-to-one corresponding relationship is as follows:

{\bf Case} (1):
For $m=1$, $B^{(1)}_1=1$.
For $m\geq2$, denote $m= 2^{j}+k$, $j\geq1$  and $ 0\leq k<2^{j}$.

{\bf Case} (2):
For $m=1$, $B^{(2)}_1=1$.
For $m=2$, $B^{(2)}_2=z$.
For $m\geq3$, denote $m= 2^{j-1}+k+2, j\geq1$  and $ 0\leq k<2^{j-1}$.

Let $\chi(2^{j}x-k)$ be the characteristic function on the interval $[2^{-j}k, 2^{-j}(k+1)]$.
For the above three TM basis, the first basis has no locality property and is not an unconditional basis for $p\neq 2$.
The following Theorem is known.
\begin{Theorem} \label{th:1}
$\{B_m^{(0)}\}_{m=1}^{+\infty}$ defined in \eqref{eq:B1.1} is not unconditional basis in $\mathbb{H}^{p}(D)(p\neq 2)$.
\end{Theorem}
Although the above Theorem \ref{th:1} is well known,
we will give still a simple proof after Proposition \ref{prop:us2}
in order to compare it with the case where an unconditional basis can be formed.
For the rest two cases,
$\{B_m^{(i)}\}_{m=1}^{+\infty}$
have special vanishing moment properties and locality,
the distances to boundary and rotations satisfy special law.
Similar to wavelets, $\{B_m^{(i)}\}_{m=1}^{+\infty}$ form unconditional basis for $p\neq 2$.
Our main Theorem is
\begin{Theorem} \label{th:1111}
Given $1<p<\infty$.
For $i=1,2$, $\{B_m^{(i)}\}_{m=1}^{+\infty}$ defined in \eqref{eq:B1.2},  \eqref{eq:B1.3} and \eqref{eq:point} form unconditional basis in $\mathbb{H}^{p}(D)$.
Further, we have
\begin{equation} \label{eq:q1}
\quad\|f\|_{\mathbb{H}^{p}}\equiv \|f\|_{\mathbb{H}^{p}_{(1)}} = |f_{1}|+ \|(\sum_{j\geq1, 0\leq k<2^j} 2^{j} |\langle f, B_{j,k}\rangle|^{2}\chi(2^{j}x-k))^{\frac{1}{2}}\|_{L^{p}}<+\infty,
\end{equation}
\begin{equation} \label{eq:q2}
\quad\|f\|_{\mathbb{H}^{p}} \equiv \|f\|_{\mathbb{H}^{p}_{(2)}} = |f_1|+ |f_{2}|+ \|(\sum_{j\geq1, 0\leq k<2^{j-1}} 2^{j} |\langle f, B_{j,k}\rangle|^{2}\chi(2^{j}x-k))^{\frac{1}{2}}\|_{L^{p}}<+\infty.
\end{equation}
\end{Theorem}

The above theorems provide a rigorous mathematical basis for the use of TM systems
on complex analysis and artificial intelligence.
The rest of this paper is structured as follows:
In Section 2, we present some preliminaries on
Hardy-Littlewood maximum function and
Khintchine inequality.
In Section 3, we present first some preliminaries on TM system.
Then we apply the non-locality property of the particular TM system $B_m^{(0)}(z)$
to prove that $B_m^{(0)}(z)$ cannot become unconditional basis in $\mathbb{H}^{p}(D)(p \neq 2)$.
In Section 4, for $i=1$ and $2$ and $1<p <\infty$,
by using Khintchine inequality based on Bernoulli probability measure,
we consider some sufficient and necessary condition for
$B_m^{(i)}(z)$ to become unconditional basis in $\mathbb{H}^{p}(D)$.
In Section 5, we prove two properties related to TM system
where the rational functions are restricted on the boundary of the unit disk $D$.
In the last Section, we use Hardy-Littlewood maximum function
and Fefferman-Stein vector maximum function to prove Main Theorem \ref{th:1111}.

\section{Maximum function and Khintchine inequality}
For $1<p<\infty$, Hardy space can be characterized by the $L^{p}$ norm on the boundary.
\begin{Definition} \label{def:1}
For $1<p<\infty$, denote
\begin{equation}
 \mathbb{H}_{\partial}^{p}( D)= \{ F(z): F(z)= F(|z| e^{2\pi i t}) \text{ is analytic in disk } D \text{ and}  \int^{1}_{0} |F( e^{2\pi i t})|^{p} dt<\infty\}.
\end{equation}
\end{Definition}
The following Proposition says that the analytic Hardy spaces $\mathbb{H}^{p}( D)$ defined by norm in the unit disk
and the analytic Hardy spaces $\mathbb{H}_{\partial}^{p}( D)$ defined by norm on the boundary of the unit disk are the same spaces.
See \cite{AN, B, CDL}.
\begin{Proposition} \label{prop:DB}
For $1<p<\infty$,
$$\|f\|_{\mathbb{H}^p(D)} \sim \|f|_{\partial D}\|_{L^p}.$$
Hence
\begin{equation}
 \mathbb{H}^p(D)=\mathbb{H}^p_{\partial}(D).
\end{equation}
\end{Proposition}

When comparing the norm of a function in $\mathbb{H}^{p}(D)$,
we need to use Hardy-Littlewood maximum function.
We use Hardy-Littlewood maximum function to consider Hardy spaces $\mathbb{H}^{p}(D)$.
See \cite{CDL, Triebel} and \cite{Yang1}.
Hardy-Littlewood maximum function is defined as follows.

\begin{Definition}
Let $f(x)$ be a locally integrable function on $\mathbb{R}$.
Let $Q$ denote the interval on $\mathbb{R}$ with length $|Q|$.
For any point $x \in \mathbb{R}$, defined
$$Mf(x)= \sup\limits_{x\in Q}\frac{1}{|Q|}\int_{Q} |f(y)|dy,$$
where $Mf$ is the Hardy-Littlewood maximum function of $f(x)$.
Similarly, we can defined Hardy-Littlewood maximum function on the interval $[0,1]$:
$$Mf(x)= \sup\limits_{|Q|\leq 1, x\in Q}\frac{1}{|Q|}\int_{Q} |f(y)|dy .$$
\end{Definition}

For the function column $f_k$, denote $\|\{f_k\}\|_{L^p(l^q)}=\|(\sum_k |f_k|^q)^{\frac{1}{q}}\|_{L^p}$.
Hardy spaces $\mathbb{H}^{p}$ are spaces of type $L^{p}(l^{2})$, see \cite{CDL, FS, Triebel}.
Fefferman-Stein vector value maximum operator theorem told us that
the norm defined by a maximal operator is controlled by the non-maximal case.
\begin{Lemma} \label{lemma:1}
For $1<p,q<\infty$, then
\begin{equation}
\|\{Mf_k\}\|_{L^p(l^q)} \leq \|\{f_k\}\|_{L^p(l^q)}, \,\forall f=\{f_k\}.
\end{equation}
\end{Lemma}

In order to consider the properties of an unconditional basis,
we need to use the Khintchine inequality, which has to do with probability.
See \cite{Zygmund}.
Let $\Lambda$ be a set of indices
and let $\Omega$ be the product set $\{1, -1\}^{\Lambda}$.
We provide this product set a Bernoulli probability measure $d u(\omega)$.
This measure is defined by the product of measures obtained by taking mass $\frac{1}{2}$ at $-1$ and $1$ for each factor.
Thus the element over $\Omega$ is the sequence $\omega(\lambda) (\lambda\in\Lambda)$ consisting of $\pm 1$.
We need the following Lemma on Khintchine inequality.

\begin{Lemma} \label{lemma:2}
On the closed subspace of $L^2(\Omega)$ space
composed of functions $S(\omega)=\sum\limits_{\lambda\in\Lambda}\alpha(\lambda)\omega(\lambda)$,
all $L^p(\Omega,d u(\omega))(0<p<\infty)$ norms are equivalent.
That is, there exist two constants $C_p \geq C_p^{'} > 0$ such that
\begin{equation}
C_p^{'} (\sum_{\lambda \in \Lambda} |\alpha(\lambda)|^2)^{\frac{1}{2}} \leq  (\int_{\Omega}|S(\omega)|^{p} d u(\omega))^{\frac{1}{p}}
\leq C_p (\sum_{\lambda \in \Lambda} |\alpha(\lambda)|^2)^{\frac{1}{2}}.
\end{equation}
\end{Lemma}

\section{Preliminaries on TM system}
We consider Szeg\"{o} kernel $e_{a}(z)$ with the $L^2$ unit modularized.
Let $e_{a}(z)$ be rational function at point $a$ within the unit disk $D$
defined as follows:
$$e_{a}(z)= \frac{\sqrt{1-|a|^{2}}}{1-\bar{a} z}.$$
Let $\prod\limits_{l=1}^{r-1}\frac{z-a_l}{1-\bar{a}_l z}$ denote a Blaschke product of the $(r-1)$ order.
TM orthogonal system $\{B_m\}^{+\infty}_{m=1}$ is made up of both Szeg\"{o} kernel and Blaschke product.
Takenaka \cite{Ta} proved that
$\overline{span}\{B_m\}^{+\infty}_{m=1}=\mathbb{H}^2(D)$ holds if and only if $a_m$ satisfies the hyperbolic non-separability condition
$$\sum_{m=1}^{+\infty} (1-|a_m|)=+\infty.$$
Qian-Chen-Tan \cite{QCT} generalized such result for $p\in (1,\infty)$
and proved that
$\overline{span}\{B_m\}^{+\infty}_{m=1}=\mathbb{H}^p(D)$ holds if and only if $a_m$ satisfies the hyperbolic non-separability condition
$$\sum_{m=1}^{+\infty} (1-|a_m|)=+\infty.$$

A special case of TM system is that $a_m=0$ for all $m$.
In this case $B_m(z)=z^m$.
Combining the conclusions in \cite{QCT} and \cite{Ta}, we have the following property.

\begin{Proposition}\label{prop:us2}
$\{1, z, \cdots, z^m, \cdots\}$ forms an unconditional basis for Hilbert space $\mathbb{H}^{2}(D)$
and a Schauder basis for Banach space $\mathbb{H}^{p}(D)(1 < p < \infty)$.
\end{Proposition}
However, the system $\{1, z, \cdots, z^m, \cdots\}$ has no locality
and cannot form unconditional basis for $\mathbb{H}^p(D)(p \neq 2)$.
Theorem \ref{th:1} is well known,
but in order to compare with the latter two bases and see how locality works on unconditional basis properties,
we still use Khintchine's inequality and proof by contradiction to give a short proof of Theorem \ref{th:1}.
\begin{proof}
For the analytic function $F(z)$ on the unit disk, we consider its value $f(x)$ on the unit circle
$f(x)=F(e^{2\pi ix})=\sum_{k\geq 0} \alpha(k)e^{2\pi ikx}$.
For $\{\omega(k)\}_{k\geq 0}\in\{-1, 1\}^{\mathbb{N}}$, denote $T_{\omega}e^{2\pi ikx}=\omega(k) e^{2\pi ikx}$.
All the functions $S(\omega)=\sum_{k \geq 0}\alpha(k)\omega(k)e^{2\pi ikx}$
belong to the closed subspace in $L^2(\Omega)$ space.
According to Lemma \eqref{lemma:2},
all modulus of $L^p(\Omega,d u(\omega))$
are equivalent to each other for $0<p<\infty$.
Denote $N_{p}(x)=(\int_{\Omega}|S(\omega)|^p d u(\omega))^{\frac{1}{p}}$.
According to the unconditional basis assumption and
Khintchine inequality,  we have

\begin{equation}\label{eq:3.1}
C_1(\sum\limits_{k \geq 0} |\alpha(k)e^{2\pi ikx}|^2)^{\frac{1}{2}} \leq N_{p}(x) \leq
 C_2(\sum\limits_{k \geq 0} |\alpha(k)e^{2\pi ikx}|^2)^{\frac{1}{2}}.
\end{equation}

According to the  formula \eqref{eq:3.1}, we have
\begin{equation}\label{eq:3.2}
N_{p}(x)\sim (\sum\limits_{k \geq 0} |\alpha(k)|^2)^{\frac{1}{2}}.
\end{equation}
That is to say, $\mathbb{H}^p(D)=\mathbb{H}^2(D)$.
But for $p\neq 2$, $\mathbb{H}^p(D)\neq \mathbb{H}^2(D)$,
this leads to a paradoxical conclusion.
Thus, by contradiction,
$\{1, z, \cdots, z^m, \cdots\}$ cannot form an unconditional basis for $\mathbb{H}^p(D)(p \neq 2)$.

\end{proof}

For TM system composed by system of trigonometric functions,
the orthogonal TM basis cannot form an unconditional basis for $\mathbb{H}^p(D)(p \neq 2)$.

\section{Unconditional basis and Bernoulli probability measure}

Unconditional basis has closed relation with Bernoulli probability measure.
We consider such relation via Proposition \ref{prop:DB}.
For case (1) and $\{\omega(k)\}_{k\geq 0}\in\{-1, 1\}^{\mathbb{N}}$, let $T_{\omega_{1}}B_{1}^{(1)}=\omega_{1}B_{1}^{(1)}$
and for $j\geq 1, 0\leq k< 2^{j}$ and $m= 2^{j}+k$,
let $T_{m}B_{m}^{(1)}= T_{\omega_{j,k}}B_{j,k}^{(1)}=\omega_{j,k}B_{j,k}^{(1)} =\omega_{m}B_{m}^{(1)}$.
Denote $T_{\omega}f= \omega_{1} \langle f, B_{1}^{(1)}\rangle B_{1}^{(1)} +   \sum\limits_{j\geq1, 0\leq k<2^{j}}\langle f, B_{j,k}^{(1)}\rangle T_{\omega_{j,k}}B_{j,k}^{(1)}$.
Denote $$N_p^{(1)}f=|f_{1}|+ \|(\sum\limits_{j\geq1, 0\leq k<2^j}  |\langle f,B_{j,k}^{(1)}\rangle|^{2}|B_{j,k}^{(1)}|^2)^{\frac{1}{2}}\|_{L^{p}([0,1])}.$$
We apply  Khintchine inequality to consider unconditional basis.

\begin{Theorem} \label{TH:3} For Case (1) and $1<p<\infty$, we have
\begin{equation} \label{eq:14}
\|f\|_{\mathbb{H}^p} \simeq N^{(1)}_pf.
\end{equation}
\end{Theorem}
\begin{proof}
According to Proposition \ref{prop:DB},
the fact that $B_{j,k}^{(1)}$ is an unconditional basis for $\mathbb{H}^p(D)$
is equivalent to the following two conditions:

\begin{equation}\label{4.1}
\|T_{\omega}f|_{\partial D}\|_{L^p} \leq \|f|_{\partial D}\|_{L^p}, \forall \omega\in \{-1, 1\}^{\mathbb{N}}.
\end{equation}
\begin{equation}\label{4.2}
\|f|_{\partial D}\|_{L^p} \leq \|T_{\omega}f|_{\partial D}\|_{L^p}, \forall \omega\in \{-1, 1\}^{\mathbb{N}}.
\end{equation}

(i) We raise both sides of Eq. \eqref{4.1} to the $p$ power
$$\begin{array}{rcl} \|\{\omega_{1} \langle f, B_{1}^{(1)}\rangle B_{1}^{(1)} + \sum\limits_{j\geq1, 0\leq k<2^{j}}
\langle f, B_{j,k}^{(1)}\rangle \omega_{j,k}B_{j,k}^{(1)}\} |_{\partial D } \|_{L^p}^p
&\leq&\|f |_{\partial D } \|_{L^p}^p.
\end{array}$$

Then average the $\omega\in\Omega$ on the left side, we get
$$\begin{array}{rcl}
\int_{\Omega}\int_{[0,1]}\{\omega_{1} \langle f, B_{1}^{(1)}\rangle B_{1}^{(1)}
+ \sum\limits_{j\geq1, 0\leq k<2^{j}} \omega_{j,k} \langle f, B_{j,k}^{(1)}\rangle B_{j,k}^{(1)} \}^{p}dxdu(\omega)
&\leq&\|f |_{\partial D } \|_{L^p}^p.
\end{array}$$

Thus, we get the double integral with the measure  $dxdu(\omega)$  over the interval $[0,1]\times \Omega$.
We use Fubili theorem,
$$\begin{array}{rcl}
\int_{[0,1]} \int_{\Omega} \{\omega_{1} \langle f, B_{1}^{(1)}\rangle B_{1}^{(1)}
+ \sum\limits_{j\geq1, 0\leq k<2^{j}} \omega_{j,k} \langle f, B_{j,k}^{(1)}\rangle B_{j,k}^{(1)} \}^{p}dxdu(\omega)
&\leq&\|f |_{\partial D } \|_{L^p}^p.
\end{array}$$

By right side of Khintchine inequality,
$$\begin{array}{rcl}
&&\int_{[0,1]}\int_{\Omega} |\omega_{1} \langle f, B_{1}^{(1)}\rangle B_{1}^{(1)} +
\sum\limits_{j\geq1, 0\leq k<2^{j}}\langle f, B_{j,k}^{(1)}\rangle \omega_{j,k}B_{j,k}^{(1)}|^{p}du(\omega)dx\\
&\leq& C\int_{[0,1]}(|\langle f, B_{1}^{(1)}\rangle|^{2} +\sum\limits_{j\geq1, 0\leq k<2^{j}}|\langle f, B_{j,k}^{(1)}\rangle B_{j,k}^{(1)}|^2)^{\frac{p}{2}}dx\\
&=& C N^{(1)}_pf\\
&\leq & \|f |_{\partial D }\|_p^p.
\end{array}$$

That is to say, we have $(N_p^{(1)})^p \leq c\|f\|_p^p$.

(ii) We raise both sides of Eq. \eqref{4.2} to the $p$ power,
$$\begin{array}{rcl}
\|f |_{\partial D } \|_{L^p}^p
\leq
\|\{ \omega_{1}\langle f, B_{1}^{(1)}\rangle B_{1}^{(1)} + \sum\limits_{j\geq1, 0\leq k<2^{j}}
\omega_{j,k} \langle f, B_{j,k}^{(1)}\rangle B_{j,k}^{(1)}\} |_{\partial D } \|_{L^p}^p.
\end{array}$$

Then average $\omega\in\Omega$ on the right side,
we get the double integral with the measure  $dxdu(\omega)$  over the interval $[0,1]\times \Omega$.
$$\begin{array}{rcl}
\|f |_{\partial D } \|_{L^p}^p
\leq\int_{\Omega}\int_{[0,1]}\{ \omega_{1}\langle f, B_{1}^{(1)}\rangle B_{1}^{(1)}
+ \sum\limits_{j\geq1, 0\leq k<2^{j}}\omega_{j, k} \langle f, B_{j,k}^{(1)}\rangle B_{j,k}^{(1)} \}^{p}dxdu(\omega).
\end{array}$$

We switch the order of integration with respect to $x$ and $\omega$.
By Fubili theorem and get
$$\begin{array}{rcl}
\|f |_{\partial D } \|_{L^p}^p \leq
\int_{[0,1]}\int_{\Omega}\{\omega_{1} \langle f, B_{1}^{(1)}\rangle B_{1}^{(1)}
+ \sum\limits_{j\geq1, 0\leq k<2^{j}} \omega_{j, k}\langle f, B_{j,k}^{(1)}\rangle B_{j,k}^{(1)} \}^{p}dxdu(\omega)
\end{array}$$

By the left side of Khintchine inequality, we get
$$\begin{array}{rcl}
&&\int_{[0,1]}\int_{\Omega}\{\omega_{1} \langle f, B_{1}^{(1)}\rangle B_{1}^{(1)}
+ \sum\limits_{j\geq1, 0\leq k<2^{j}}\omega_{j,k} \langle f, B_{j,k}^{(1)}\rangle B_{j,k}^{(1)} \}^{p}dxdu(\omega)\\
&\geq&
C\int_{[0,1]}(|\langle f, B_{1}^{(1)}\rangle|^{2} +\sum\limits_{j\geq1, 0\leq k<2^{j}}|\langle f, B_{j,k}^{(1)}\rangle B_{j,k}^{(1)}|^2)^{\frac{p}{2}}dx\\
&=&c(N_p^{(1)})^p\\
&\geq&\|f|_{\partial  D}\|_p^p.
\end{array}$$

That is to say, we have
$$ \|f\|_p^p\leq c(N_p^{(1)})^p.$$
Hence \eqref{eq:14} holds.
\end{proof}

For case (2), let $T_{\omega_{1}}B_{1}^{(2)}=\omega_{1}B_{1}^{(2)}$
and $T_{\omega_{2}}B_{2}^{(2)}=\omega_{2}B_{2}^{(2)}$.
For $j\geq 1, 0\leq k < 2^{j-1}$  and $m= 2^{j-1}+k+2$,
let $T_{\omega_{j,k}}B_{j,k}^{(2)}=T_{\omega_{m}}B_{m}^{(2)}=\omega_{m}B_{m}^{(2)}=\omega_{j,k}B_{j,k}^{(2)}$.
We have $$T_{\omega}f=\omega_{1} \langle f, B_{1}^{(2)}\rangle B_{1}^{(2)} + \omega_{2} \langle f, B_{2}^{(2)}\rangle B_{2}^{(2)} +
\sum\limits_{j\geq1, 0\leq k<2^{j-1}}\langle f, B_{j,k}^{(2)}\rangle T_{\omega_{j,k}}B_{j,k}^{(2)}.$$
Denote
$$N_p^{(2)}=|f_{1}|+ |f_{2}|+ \|(\sum\limits_{j\geq1, 0\leq k<2^{j-1}}  |\langle f,B_{j,k}^{(2)}\rangle|^{2}|B_{j,k}^{(2)}|^2)^{\frac{1}{2}}\|_{L^{p}([0,1])}.$$
Similar to Case (1), we apply  Khintchine inequality to prove that $B_{j,k}^{(2)}$ is also an unconditional basis.

\begin{Theorem} \label{TH:4} For Case (2) and $1<p<\infty$, we have
\begin{equation} \label{eq:th2}
\|f\|_{\mathbb{H}^p} \simeq \|N^{(2)}_pf\|_{L^p}.
\end{equation}
\end{Theorem}

\begin{proof}
According to Proposition \ref{prop:DB},
the fact that $B_{j,k}^{(2)}$ is an unconditional basis for $\mathbb{H}^p(D)$ is equivalent to the following two expressions:
\begin{equation}\label{4.3}
\|T_{\omega}f|_{\partial D}\|_{L^p} \leq \|f|_{\partial D}\|_{L^p}, \forall \omega\in \{-1, 1\}^{\mathbb{N}}.
\end{equation}
\begin{equation}\label{4.4}
\|f|_{\partial D}\|_{L^p} \leq \|T_{\omega}f|_{\partial D}\|_{L^p}, \forall \omega\in \{-1, 1\}^{\mathbb{N}}.
\end{equation}

Similar to the case (1), we can use the same method to prove that Theorem \ref{TH:4} holds.
Eq. \eqref{4.3} is similar to Eq. \eqref{4.1}.
Applying similar proof, we obtain
$$\|T_{\omega}f^{(2)}|_{\partial D}\|_{L^p}^p\leq  (N_p^{(1)})^p \leq c\|f\|_p^p.$$
Eq. \eqref{4.4} is similar to Eq. \eqref{4.2}. By similar proof, we obtain
$$ \|f\|_p^p\leq c\|T_{\omega}f^{(2)}|_{\partial D}\|_{L^p}^p\leq c(N_p^{(2)})^p.$$
Hence Eq. \eqref{eq:th2} holds.

\end{proof}

\section{Some properties of $B_{j,k}(e^{2\pi i x})$}

On the boundary of unit disk $D$,  we have $z=e^{2\pi i x}$.
\begin{Lemma}\label{pro:3}
The following two properties hold for $B_{j,k}(e^{2\pi i x})$.\\
(i) There exist a positive constant $c$ such that $\forall j\geq 1, 0\leq k< 2^{j}, x\in [2^{-j}k, 2^{-j}(k+1)]$,
\begin{equation}
|B_{j,k}(e^{2\pi i x})|^2 \geq c2^j\chi(2^jx-k).
\end{equation}
(ii) Further, there exist a positive constant $C$ such that, for $j\geq1$,
\begin{equation}
 \sum_{0\leq k<2^j}|B_{j,k}(e^{2\pi i x})|^2  \leq C2^{j}.
\end{equation}
\end{Lemma}

\begin{proof}
On the boundary of disk $D$,   we have $|e^{2\pi i x}|^2 = 1$.
Hence
$$\prod\limits_{l=2}^{m}|\frac{e^{2\pi i x}-a_l}{1-\bar{a}_le^{2\pi i x}}|^2= \prod\limits_{l=2}^{m}|\frac{e^{2\pi i x}-a_l}{\overline{e^{2\pi i x}}-\bar{a}_l}|^2|\overline{e^{2\pi i x}}|^2=1.$$

(i) Given  $j\geq1$, $ 0\leq k<2^{j}$ and $m=2^{j}+k$.
For $ x\in [2^{-j}k, 2^{-j}(k+1)]$, we have $|e^{2\pi i(x-h_{j,k})}|= 1$.
Hence
$$\begin{array}{rcl}
|B_{j,k}^{(1)}(e^{2\pi i x})|^2&=&
| \frac{\sqrt{1-|a_{j,k}|^{2}}e^{2\pi i x}}{1-\bar{a}_{j,k} e^{2\pi i x}} \prod\limits_{l=2}^{m}\frac{e^{2\pi i x}-a_l}{1-\bar{a}_le^{2\pi i x}}|^2\\
&=&
| \frac{(1-|a_{j,k}|^{2})|e^{2\pi i x}|^2}{(1-\bar{a}_{j,k} e^{2\pi i x})^2}| \prod\limits_{l=2}^{m}|\frac{e^{2\pi i x}-a_l}{1-\bar{a}_le^{2\pi i x}}|^2\\
&=&
\frac{|1-r_j^2|}{|1-r_je^{2\pi i(x-h_{j,k})}|^2} .
\end{array}$$
For $t\rightarrow0$, $1-\cos t\sim \frac {t^2}{2}$, we have
$$\begin{array}{rcl} \frac{1}{|1-r_{j} e^{2\pi i(x-h_{j,k})}|}
&=& \frac{1}{\sqrt{(1-r_{j}\cos(2\pi (x-h_{j,k})))^2 +(r_{j} \sin(2\pi (x-h_{j,k})))^2}}\\
&=&\frac{1}{\sqrt{(1-r_{j})^2 +2r_{j}(1 -\cos(2\pi (x-h_{j,k})))}}\\
&=&\frac{1}{\sqrt{( 2^{-j})^2 +4\pi^2r_{j} (x-h_{j,k})^2}}\\
&=&\frac{1}{2^{-j}\sqrt{1+4\pi^2r_{j}(2^jx-k)^2}}\\
&\geq&c2^{j}.
\end{array}$$

We know $|1-r_{j}^2 |= 1- (1- 2^{-j})= 2^{-j}$, hence
$$\begin{array}{rcl}
|B_{j,k}^{(1)}(e^{2\pi i x})|^2&=&
\frac{|1-r_j^2 |}{|1-r_je^{2\pi i(x-h_{j,k})}|^2}  |z|^2   \prod\limits_{l=2}^{m}|\frac{z-a_l}{1-\bar{a}_lz}|^2\\
&\geq&c2^{j}.
\end{array}$$

(ii)
For $j\geq1$ , $ 0\leq k<2^{j}$ and $m=2^{j}+k$, we have

$$\begin{array}{rcl}
\sum\limits_{0\leq k<2^j}|B_{j,k}^{(1)}(e^{2\pi i x})|^2
&=&
\sum\limits_{0\leq k<2^j}\frac{|1-r_j^2||e^{2\pi i x}|^2}{|1-r_je^{2\pi i(x-h_{j,k})}|^2}     \prod\limits_{l=2}^{m}|\frac{e^{2\pi i x}-a_l}{1-\bar{a}_le^{2\pi i x}}|^2\\
&=&
\sum\limits_{0\leq k<2^j}\frac{|1-r_j^2|}{|1-r_je^{2\pi i(x-h_{j,k})}|^2}.
\end{array}$$

Further, we have
$$\begin{array}{rcl} \frac{1}{|1-r_{j} e^{2\pi i (x-h_{j,k})}|^2}
&=& \frac{1}{|(1-r_{j}\cos(2\pi (x-h_{j,k})))^2 +(r_{j} \sin(2\pi (x-h_{j,k})))^2|}\\
&=&\frac{1}{|(1-r_{j})^2 +2r_{j}(1 -\cos(2\pi (x-h_{j,k})))|}\\
&\leq &\frac{1}{|(2^{-j})^2 +4\pi^2r_{j} (x-h_{j,k})^2|}\\
&\leq & C 2^{2j}\frac{1}{1+(2^{j}x-k)^2}.
\end{array}$$

Hence
$$\begin{array}{c}
\sum\limits_{0\leq k<2^j}|B_{j,k}^{(1)}(e^{2\pi i x})|^2
\leq  C2^j\sum\limits_{0\leq k<2^j}\frac{1}{1+(2^{j}x-k)^2} \leq C 2^j.
\end{array}$$

\end{proof}

\section{Proof of Main Theorem \ref{th:1111}}

For case (1), let $$f_j=\sum_{0 \leq k < 2^j} 2^{\frac{j}{2}}|\langle f,B_{j,k}\rangle|\chi(2^j x-k),
bf_j=\sum_{0 \leq k < 2^j} |\langle f,B_{j,k}\rangle|^2 |B_{j,k}|^2.$$
For case (2), let $$f_j=\sum_{0 \leq k < 2^{j-1}} 2^{\frac{j}{2}}|\langle f,B_{j,k}\rangle|\chi(2^j x-k),
bf_j=\sum_{0 \leq k < 2^{j-1}} |\langle f,B_{j,k}\rangle|^2 |B_{j,k}|^2.$$
In both cases, denote $bf=(\sum_{j \geq 1} bf_j)^{\frac{1}{2}}$ and denote $Sf=(\sum_{j \geq 1} |f_j|^2)^{\frac{1}{2}}$.

Now, we come to prove Main theorem \ref{th:1111}.

\begin{proof}

(i) $\forall x \in [2^{-j}k, 2^{-j}(k+1)]$, by (i) of Lemma \ref{pro:3}, we have $|B_{j,k}|^2 \geq c 2^j \chi(2^j x-k)$.
Further, by applying Theorem \ref{TH:3}, we have
\begin{equation}
\|f\|_{H^p_{(1)}} \lesssim  N_p^{(1)} \lesssim \|f\|_p.
\end{equation}
By applying Theorem \ref{TH:4}, we have
\begin{equation}
\|f\|_{H^p_{(2)}} \lesssim N_p^{(2)} \lesssim \|f\|_p.
\end{equation}
(ii)
From the property of maximal function, we have
$\forall x \in [2^{-j}k, 2^{-j}(k+1)]$, $2^\frac{j}{2}|\langle f, B_{j,k}\rangle| \leq Mf_j.$
Thus,
$$\sum_{0 \leq k < 2^j}|\langle f, B_{j,k}\rangle|^2|B_{j,k}|^2 \leq
c 2^{-j} (Mf_j)^2 \sum_{0 \leq k < 2^j}|B_{j,k}|^2 \leq c|Mf_j|^2.$$
Hence, for case (1), we have
\begin{equation} \label{meq:4.1}
N_p^{(1)} \lesssim |f_1|+\|(\sum_{j \geq 1}(Mf_j)^2)^{\frac{1}{2}}\|_{L^p}.
\end{equation}
According to the Fefferman-Stein vector maximum theorem, for $1<p<\infty$,
the right-hand side of equation \eqref{meq:4.1} is less than $ |f_1|+\|(\sum_{j \geq 1}(f_j)^2)^{\frac{1}{2}}\|_{L^p}$.
Thus,
$$N_p^{(1)} \leq c\|f\|_{H^p_{(1)}}.$$

For case (2), we have
\begin{equation} \label{meq:4.2}
N_p^{(2)}\lesssim |f_{1}|+ |f_{2}|+ \|(\sum_{j \geq 1}(Mf_j)^2)^{\frac{1}{2}}\|_{L^p}.
\end{equation}

According to the Fefferman-Stein vector maximum theorem, for $1<p<\infty$,
the right-hand side of equation \eqref{meq:4.2} is less than $ |f_1|+|f_2|+\|(\sum_{j \geq 1}(f_j)^2)^{\frac{1}{2}}\|_{L^p}$.
Thus,
$$N_p^{(2)} \leq c\|f\|_{H^p_{(2)}}.$$

\end{proof}

\textbf{Acknowledgements.
This project is partially supported by
research grant of Macau University of Science and Technology (FRG-22-075-MCMS),
Macau Government Research Funding (FDCT0128/2022/A),
Science and Technology Development Fund of Macau SAR (005/2022/ALC),
Science and Technology Development Fund of Macau SAR (0045/2021/A),
Macau University of Science and Technology (FRG-20-021-MISE).
}

\bigskip
\noindent Haibo Yang


\medskip
\noindent
Macau Institute of Systems Engineering, \\
Macau University of Science and Technology, Macau, 999078, China.\\
\smallskip
\noindent{\it E-mail address}:
\texttt{yanghb97@qq.com}

\bigskip
\noindent Chitin Hon

\medskip
\noindent
Macau Institute of Systems Engineering, \\
Macau University of Science and Technology, Macau, 999078, China.

\smallskip
\noindent{\it E-mail address}:
\texttt{cthon@must.edu.mo}

\bigskip
\noindent Qixiang Yang

\medskip
\noindent
School of Mathematics and Statistics, Wuhan University,
Wuhan, 430072 China.

\smallskip
\noindent{\it E-mail address}:
\texttt{qxyang@whu.edu.cn}

\bigskip
\noindent Tao Qian

\medskip
\noindent
Director of Macau Centre of of Mathematical Sciences, \\
Faculty of Innovation Engineering,\\
Macau University of Science and Technology

\smallskip
\noindent{\it E-mail address}:
\texttt{tqian1958@gmail.com}

\end{document}